\documentclass[11pt]{amsart}
\usepackage[english]{babel}
\usepackage{amssymb,amsmath,amsthm}
\usepackage[latin2]{inputenc}
\usepackage{amsfonts, verbatim, amscd}
\usepackage[foot]{amsaddr}

\reversemarginpar
\newcommand\dist{\mathrm{dist}}
\newtheorem{theorem}{Theorem}
\newtheorem{proposition}[theorem]{Proposition}

\newtheorem{lemma}[theorem]{Lemma}

\theoremstyle{definition}

\def\acknowledgment{\par\addvspace{17pt}\small\rmfamily
\trivlist\if!\ackname!\item[]\else
\item[\hskip\labelsep
{\bfseries\ackname}]\fi}

\def\C{\mathbb{C}}
\def\R{\mathbb{R}}

\def\N{\mathbb{N}}
\def\D{\mathbb{D}}
\def\B{\mathbb{B}}

\def\OB{\mathcal{O}_B(\D,\C^n)}
\def\OBZ{\mathcal{O}_B(\overline{\D},\C^n)}
\def\WN{W^{1,p}(\D,\C^n)}
\def\LN{L^p(\D,\C^n)}

\begin{document}
\title{Solutions of Pascali systems attached to convex boundaries}

\author{Barbara Drinovec Drnov\v sek \& Uro\v s Kuzman}
\address[Barbara Drinovec Drnov\v sek]{Faculty of Mathematics and Physics (University of Ljubljana) 
and Institute of Mathematics, Physics and Mechanics, barbara.drinovec@fmf.uni-lj.si}

\address[Uro\v s Kuzman]{Faculty of Mathematics and Physics (University of Ljubljana), Faculty of Natural Sciences and Mathematics (University of Maribor) and Instit. of Mathematics, Physics and Mechanics, uros.kuzman@fmf.uni-lj.si}

\begin{abstract}
Given a bounded strictly convex domain $\Omega\Subset \mathbb{C}$ and a point $q\in \Omega$ we construct 
a continuous solution of the Pascali-type elliptic system of differential equations that is centered in $q$, maps the unit disc into $\Omega$ and the unit circle into $\partial \Omega$. 
\end{abstract}
\maketitle

Let $\mathbb{D} \subset \mathbb{C}$ be the unit disc. The \emph{Pascali system} on $\D$ is an elliptic system of differential equations which can be written in the following form:
	\begin{equation}\label{Pascali}
		\bar{\partial}_B(w)=w_{\overline{\zeta}}+B_1 w + B_2\overline{w}=0,
	\end{equation} 
where $w\colon\mathbb{D}\to \mathbb{C}^n$ is a vector function, while $B_1$ and $B_2$ are $n\times n$ matrix functions defined on the unit disc. 
We always assume $n\ge 2$.
These systems are named after D.\ Pascali, who in \cite{PASCALI} initiated their research as an analogue of the scalar theory developed by Bers \cite{BERS} and Vekua \cite{VEKUA}. Their solutions are often studied as a subclass of generalized analytic vectors corresponding to systems with vanishing Beltrami coefficients \cite{BOJARSKI, GB}. Pascali systems appear in fluid dynamics (see \cite{WENDLAND}) and can be obtained as a linearization of the generalized Cauchy-Riemann system along a $J$-holomorphic map (see \cite{ST}). 

In this paper, we construct solutions for Pascali systems that obey certain nonlinear boundary conditions and
contribute a geometric approach from complex analysis to a widely studied field of boundary value problems for elliptic systems in the complex plane (see, e.g., \cite{BEGEHRWEN2} and the references therein for the scalar case and \cite{MANJA} for $n>2$). 
The present research is a continuation of our work in \cite{BDDKUZ}, where analogues of classical approximation theorems for holomorphic functions are discussed for solutions of Pascali systems.  

Given a Pascali system \eqref{Pascali}, we denote the set of its solutions by $\OB$. 
In general, the regularity of its elements depends on the regularity of $B_1$ and $B_2$. 
Throughout the paper, we assume that the coefficients of $B_1$ and $B_2$ are of class $\mathcal{C}^\infty$ on $\D$;
therefore, the elements of $\OB $ are also smooth on $\D$. Furthermore, we denote by $\OBZ$ the set of solutions that are continuous up to the boundary and say that a map $w\colon \D\to \C^n$ is \emph{centered in} $q\in\C^n$ if $w(0)=q$.

Our main construction is contained in the proof of the following theorem.
\begin{theorem}\label{main}
Let $\Omega\Subset \mathbb{C}^n$ be a smoothly bounded strictly convex domain and $q\in \Omega$. There exists a map $u\in \OBZ$ centered in $q$ and such that $u(\D)\subset \Omega$, $u(\partial\D)\subset \partial\Omega$. In particular, the map $u\colon \D\to\Omega$ is proper.
\end{theorem}

A smoothly bounded strictly convex domain $\Omega\Subset \mathbb{C}^n$ is a sublevel set of some strictly convex function $\rho\in\mathcal C^{\infty}(\overline\Omega)$, i.e., $\Omega=\{\rho<0\}$, therefore Theorem \ref{main} provides a solution to the following nonlinear boundary value problem:
$$\left\{\begin{array}{ll}
\bar{\partial}_B(u)=0,&\\
\rho (u(\zeta))<0, &\zeta\in \D,\\
\rho(u(\zeta))=0, &\zeta\in \partial\D,\\
u(0)=q.
\end{array}\right.$$
In the proof, we start with a small nonconstant solution of the Pascali system centered in $q$ provided by Lemma \ref{small disc}. Then, we push the boundary of the solution to higher levels of the function $\rho$, obtaining the desired solution in the limit.
In the construction, we use an approximate solution of a certain nonlinear Riemann-Hilbert boundary value problem provided by Lemma \ref{RHlemma}: it provides a deformation of the given solution in the direction of a given vector field together with a precise control on the geometric placement of the new solution. Such an approach is standard in various constructions of proper holomorphic and pseudoholomorphic maps (see  \cite{FFJG,BDDFF,PROPER} and the references therein). More recently, this approach has been extensively used in the minimal surface theory in obtaining solutions to the conformal 
Calabi-Yau problem on the existence of open Riemann surfaces that are the complex structures of complete bounded minimal surfaces in $\R^3$,
see \cite{AFL}.

\vskip 0.2 cm
\noindent{\bf{Acknowledgements.}} The authors are supported in part by the research program P1-0291 and grants J1-3005, N1-0137, N1-0237, US/22-24-079 from ARIS, Republic of Slovenia. 

\section{On solutions of Pascali systems}
Throughout this section let $p>2$. If the map $w$ is at least $L^1_{\mathrm {loc}}$-regular the Pascali-type equation $(\ref{Pascali})$ may be written in the following weak form  
\begin{equation}\label{weak Pascali}\iint_{\D}\;\left(w^T\cdot \varphi_{\bar{\zeta}}- (B_1w+B_2\overline{w})^T\cdot \varphi\right)\;\; d\zeta d\overline{\zeta} =0,\end{equation}
where this condition is supposed to be fulfilled for all smooth vector functions $\varphi$ with compact support in $\D$. As mentioned in the introduction, by the standard bootstrapping argument, this weak form and the smoothness of $B_1$ and $B_2$ imply that the coefficients of $w$ are of class $\mathcal{C}^\infty$ on $\D$. We denote by $\mathcal{C}^{\infty}(\D,\C^n)$ the set of such vector functions and by $\mathcal{C}(\D,\C^n)$ and $\mathcal{C}(\overline{\D},\C^n)$ the set of vector functions whose coefficients are only continuous on the corresponding domains. Moreover, we denote by $\LN$ and $\WN$ the set of vector functions with coefficients in the Lebesgue or Sobolev space, respectively. For $w\in \WN$ the equation $\bar{\partial}_J(w)=0$ may be discussed in the above weak form $(\ref{weak Pascali}).$ Furthermore, by Sobolev embedding theorem, its coefficients are H\"older continuous with exponent $1-\frac{2}{p}$ on $\overline{\D}$. Hence $\OB\cap \WN\subset \OBZ.$ That is, every $\WN$-regular solution is automatically continuous up to the boundary. 

It was proved by Sukhov and Tumanov \cite[Theorem 3.6]{ST} that the operator $\bar{\partial}_B\colon \WN\to \LN$ is surjective and admits a bounded right inverse. Let us briefly explain their construction. The classical Cauchy-Green transform 
$$T(w)(\zeta)=-\frac{1}{\pi}\iint_\D \frac{w(z)}{z-\zeta}dxdy(z)$$
is a bounded operator from $\LN$ to $\WN$ that satisfies the equation $\left[T(w)\right]_{\bar{\zeta}}=w$ on $\D$ in the weak sense. Let us denote its normalization at the origin by $T_0(w)(\zeta)=T(w)(\zeta)-T(w)(0)$. We introduce the integral operator $\Psi\colon \WN\to \WN$ given by
$$\Psi(w)=w+T_0\left(B_1w+ B_2\bar{w}\right).$$
Note that $\left[\Psi(w)\right]_{\bar{\zeta}}=0$ if $\bar{\partial}_B(w)=0$ and $\Psi(w)(0)=w(0)$. Since $\Psi$ is a sum of identity and a compact operator its Fredholm index is zero. Therefore, $\Psi$ is surjective if and only if it is injective. However, for $n\geq 2$ its kernel might be non trivial. Nevertheless, after introducing an appropriate inner product one can use the base vectors of $\ker\Psi^*$ to construct a small linear perturbation $L\colon \WN\to \WN$ for which the operator $\widehat{\Psi}=\Psi+L$ is invertible. Furthermore, we may achieve that $\left[L(w)\right]_{\bar{\zeta}}=0$ and $L(w)(0)=0$ for every $w\in \WN$. This implies the following proposition.

\begin{proposition}\cite[Section 3]{ST}\label{inverse}
There exists an invertible bounded operator $\widehat{\Psi}\colon \WN\to \WN$ such that 
$\widehat{\Psi}(w)$ is holomorphic if and only if $\bar{\partial}_B(w)=0$ and that $\widehat{\Psi}(w)(0)=w(0)$. Moreover, the map $Q_B=\widehat{\Psi}^{-1}\circ T_0$ is a bounded right inverse for the operator $\bar{\partial}_B\colon \WN\to\LN$ and satisfies $Q_B(w)(0)=0$. 
\end{proposition}
\noindent Let us remark here that the above proposition remains valid for matrix functions $B_1$ and $B_2$ with only $L^p(\D)$-regular coefficients and the weak form of the equation $\bar{\partial}_B(w)=0$. Indeed, in such a case the solutions of $(\ref{weak Pascali})$ belong to the space $\WN$ by the bootstrapping argument. All other steps remain the same. We will need this fact in the proof of Theorem 5. 

An immediate corollary of the Proposition \ref{inverse} is the existence of small non constant solutions of the Pascali system through any given point.
\begin{lemma}\label{small disc}
Let $\Omega\Subset \mathbb{C}^n$ be a bounded domain and $q\in \Omega$. There exists a non constant solution $u\in \OBZ$ centered in $q$ and such that $u(\overline{\D})\subset \Omega$.
\end{lemma}
\begin{proof} Given $V\in \mathbb{C}^n$ we define a linear holomorphic map $h(\zeta)=q+\zeta\cdot V$ and the corresponding Pascali-type solution $w=\widehat{\Psi}^{-1}(h)\in\OBZ$. Note that the $\WN$-norm of $w$ depends on the norm of $V\in\C^n$. Moreover, by Sobolev embedding theorem, the same is true for the $\mathcal{C}(\overline{\D},\C^n)$-norm of $w$.  Therefore, provided that $V$ is small enough, $w$ is the map we seek.
\end{proof}
\noindent A slightly more involved construction leads to the following lemma that will be crucial for Theorem \ref{main}. The map that we construct within is often referred as an approximate solution of the nonlinear Riemann-Hilbert boundary value problem. 
\begin{lemma}
\label{RHlemma}
Given $u\in \OBZ$, $V\in\mathcal{C}(\overline{\D},\C^n)$, $r_0\in (0,1)$ and $\epsilon>0$, there are $w\in \OBZ$ and $r'\in (r_0,1)$ such that 
\begin{enumerate}
\item[(i)] for all $\zeta\in b\D$ there is $\xi\in b\D$ such that $|w(\zeta)-(u(\zeta)+\xi\cdot V(\zeta))|<\epsilon$,
\item[(ii)] for all $\zeta\in b\D$ and $r\in [r',1)$ there is $\xi\in\overline{\D}$ such that  $$|w(r\zeta)-(u(\zeta)+\xi\cdot V(\zeta))|<\epsilon,$$
\item[(iii)]$|w(\zeta)-u(\zeta)|\le\epsilon$ for all $|\zeta|\le r'$, and 
\item[(iv)]$w(0)=u(0)$.
\end{enumerate}
\end{lemma} 

\begin{proof} The map $V$ admits a smooth approximation on $\overline{\D}$. Therefore, without loss of generality in (i) and (ii) we can assume that $w\in\mathcal{C}^{\infty}(\overline{\D},\C^n).$ 

For $N\in\mathbb{N}$ we define the map $v_N\in \WN$ given by
$$v_N(\zeta)=u(\zeta)+\zeta^N V(\zeta).$$
Note that, provided that $N$ is large enough, the map $v_N$ satisfies the conditions (i)-(iv) but is not an element of $\OBZ$. However, we have 
$$\left[\bar{\partial}_B(v_N)\right](\zeta)=\zeta^N \left(V_{\bar{\zeta}}(\zeta)+B_1(\zeta) V(\zeta)\right) + \bar{\zeta}^N B_2(\zeta) \overline{V(\zeta)}.$$
Since the norm of $\zeta^N$ in $L^p(\D)$ tends to zero as $N$ tends to infinity, the same is true for $\LN$-norm of $\bar{\partial}_B(v_N)$. Therefore, we can approximate $v_N$ on $\overline{\D}$ with a solution of the Pascali system: 
Let $Q_B$ be the right inverse of $\bar{\partial}_B$ defined in Proposition \ref{inverse}. The map we seek is
$$w_N=v_N-Q_B\left(\bar{\partial}_B (v_N)\right).$$
Indeed, since $Q_B$ is a bounded operator the $\WN$-norm of the difference $w_N-v_N$ tends to zero when 
$N$ tends to infinity.
Thus $w_N$ is also $\mathcal{C}(\overline{\D},\C^n)$-close to $v_N$. This implies that properties (i), (ii) and (iii) are satisfied for $w=w_N$ for any $N$ large enough.
 Finally, since by construction we have $Q_B\left(\bar{\partial}_B (v_N)\right)(0)=0$, we also have $w(0)=v_N(0)=u(0)$.  
\end{proof}

Given a solution $w\in\OB$ there exist a holomorphic vector function $\phi_w\colon \D\to \C^n$ and an invertible $n\times n$ matrix function $S_w$ with coefficients in $W^{1,p}(\D)$ such that $w=S_w\cdot \phi_w$. This fact is often called the \emph{Similarity principle} and was proved in \cite{BUCHANAN}. It allows the elements of $\OB$ to inherit some properties from the usual holomorphic vector functions. For instance, the zero set of a non zero solution $w\in \OB$ is discrete in $\D$.

If $n=1$ one can determine the scalar functions explicitly. Indeed, set $S_w=\exp(-T_0(B_1+B_2\frac{\overline{w}}{w}))$ and check that $\phi_w=(S_w)^{-1}w$ is holomorphic 
(by Weyl's lemma vanishing of the weak derivative with respect to $\bar{\zeta}$ is sufficient). 
Moreover, outside the zero set of $w$, the absolute value of $B_1+B_2\frac{\overline{w}}{w}$ is bounded by an uniform constant that depends only on $B_1$ and $B_2$. Hence $S_w$ and $(S_w)^{-1}$ are uniformly bounded in $\mathcal{C}(\overline{\D})$ for every $w\in \mathcal{O}_B(\overline{\D},\C)$. Thus the maximum of $|w|$ on $\overline{\D}$ is less or equal the maximum of this function on $b\D$ multiplied by an uniform constant $C\geq 1$ (see e.g. \cite[\S 4.8]{VEKUA}). 

For $n\geq 2$ the construction of $S_w$ is more subtle. We establish the following version of the maximum principle for solutions of Pascali systems.
\begin{theorem}
Let $n\geq 2$ and let the set $\mathcal{W}\subset\OBZ$ be precompact in the space $\LN$. There is $C>0$ such that for every $w\in\mathcal{W}$ we have
$$\max_{\overline{\D}}|w|\leq C\cdot\max_{b\D}|w|.$$
\end{theorem}
\begin{proof}
Let $w\in\mathcal{W}$ and let $D_w$ be a diagonal matrix function with property $D_w\cdot w=\overline{w}$. 
Note that the coefficients of the matrix function $B_2D_w$ belong to the class $L^p(\D)$. Hence by Proposition \ref{inverse} there exists an invertible bounded operator $\widehat {\Psi}_w\colon \WN\to \WN$ such that $\widehat {\Psi}_w(v)$ is holomorphic if and only if $v_{\bar{\zeta}}+(B_1+B_2D_w)v=0$.
Therefore, the matrix equation $\widehat{\Psi}_w (S_w)=\mathrm{Id}$ admits a unique solution $S_w$ which
is an invertible matrix function that satisfies
$(S_w)_{\bar{\zeta}}+(B_1+B_2D_w)S_w=0.$ 
Let $\phi_w=(S_w)^{-1}w$ and note that
$$w_{\bar{\zeta}}=(S_w)_{\bar{\zeta}}\;\phi_w+S_w (\phi_w)_{\bar{\zeta}}=-(B_1+B_2D_w)S_w\phi_w+S_w (\phi_w)_{\bar{\zeta}}.$$ 
On the other hand, since $w\in\OBZ$ we have
$$w_{\bar{\zeta}}=- B_1w-B_2\bar{w}=-(B_1+B_2D_w)w=-(B_1+B_2D_w)S_w\phi_w.$$
Since $S_w$ is invertible this implies that $\phi_w$ is a holomorphic vector function. The coefficients of $S_w$ belong to the Sobolev class $W^{1,p}(\D)$, thus by the Sobolev embedding theorem, $\phi_w$ is continuous up to the boundary.

For any $\widehat{w}\in\mathcal{W}$ sufficiently close to $w$ in the space $\LN$ we can define an invertible operator $\widehat{\Psi}_{\widehat{w}}\colon \WN\to \WN$ given by 
$$\widehat{\Psi}_{\widehat{w}}=\widehat{\Psi}_w+T_0(B_2(D_{\widehat{w}}-D_w)).$$
Note that $\widehat{\Psi}_{\widehat{w}}(v)$ is holomorphic if and only if $v_{\bar{\zeta}}+(B_1+B_2D_{\widehat{w}})v=0$. Hence one can construct $S_{\widehat{w}}$ and $\phi_{\widehat{w}}$ in a similar way as above. Moreover, one can uniformly bound the largest and the smallest eigenvalue of the matrix function $S_{\widehat{w}}$ (first one from above and the second one away from zero) by imposing a bound on the norm of $\widehat{\Psi}_{\widehat{w}}-\widehat{\Psi}_w$. Together with the precompactness assumption for $\mathcal{W}$ this leads to uniform bounds for all matrix functions $S_w$ and $(S_w)^{-1}$ for $w\in\mathcal{W}$ and the desired conclusion. 
\end{proof}

The maximum principle gives the following convergence result that will be applied in the inductive proof of the main theorem.

\begin{lemma}\label{lemmakonv}
Let $u_j\in\OBZ$ be a sequence that converges to $u$ uniformly on $b\D$ and uniformly on compact sets in $\D$. Then $u\in\OBZ$. 
\end{lemma}
\begin{proof} 
The fact that $u\in\OB$ follows from the uniform compact convergence of the sequence $u_j\in\OB$ in $\D$ and the weak form (\ref{weak Pascali}). 
Moreover, it is easy to check that the set $\left\{u_j\right\}$ is precompact in $\LN$. Therefore we have
$$\max_{\overline{\D}}\left|u_j-u_k\right|\leq C \cdot\max_{b\D}\left|u_j-u_k\right|.$$
This gives the Cauchy propery of the sequence  $u_j$ in $\mathcal{C}(\overline{\D},\C^n)$.\end{proof}
\noindent

\section{Proof of Theorem \ref{main}}
The proof consists of two parts. In the first part, we start with a small solution of the Pascali system through the given point provided by Lemma \ref{small disc} and then we push its boundary into the vicinity of $b\Omega$. In the second part, we construct a sequence of solutions of the Pascali system whose boundary points converge towards $b\Omega$ and make sure that its limit is indeed a continuous map. 

By classical results from the convex geometry, see \cite{Rauch, Blaschke}, we have the following:
if $\kappa_{\min}$ and $\kappa_{\max}$ are the minimum and the maximum principal curvatures of points in $b\Omega$
with respect to the inner normal to $b\Omega$, then for each point $z\in b\Omega$, and the inner unit normal $\nu_z$ to $b\Omega$ at $z$ we have
\begin{equation}\label{ukriv}
\overline\B\left(z+\frac{1}{\kappa_{\max}}\nu_z,\frac{1}{\kappa_{\max}}\right)\subset\overline\Omega\subset \overline\B\left(z+\frac{1}{\kappa_{\min}}\nu_z,\frac{1}{\kappa_{\min}}\right).
\end{equation}

Since  $\Omega\Subset \mathbb{C}^n$ is a smoothly bounded strictly convex domain 
there are a strictly convex defining function $\rho\in\mathcal C^{\infty}(\overline\Omega)$, and  $c$, $0<c<\frac{1}{\kappa_{\max}}$, such that
\begin{eqnarray*}
\Omega=\{z\in\overline\Omega\colon \rho(z)<0\},\ \ d\rho(z)\ne 0 \text{ on }b\Omega,\\
\text{ for any } z\in\Omega \text{ with } \rho(z)>-c \text{ we have }\dist(z,b\Omega)=-\rho(z).
\end{eqnarray*}
We obtain such a defining function by gluing the signed distance function $\dist(\cdot,\Omega)-\dist(\cdot,\C^n\setminus\Omega)$, which is smooth and strictly convex near the boundary, with an appropriate smooth strictly convex function in the interior.  We denote by $p_0$ the minimum of $\rho$ in $\overline\Omega$, 
and for $\eta\in [\rho(p_0),0)$
by $\Omega_\eta$ the level set $\{z\in\Omega \colon \rho(z)=\eta\}$. If $q\ne p_0$, the level set $\Omega_{\rho(q)}$ is a smooth strictly convex hypersurface.

Since $1-c\kappa_{\min}\in (0,1)$ we can choose $\alpha\in (\frac 1 2,1)$, so close to $\frac 1 2$ that 
\begin{equation}\label{defin_d}
d=2\alpha (1-c\kappa_{\min})\in (0,1).
\end{equation}
Now we choose $\delta_1\in (0,c)$, such that 
\begin{equation}\label{defin_delta0}
\text{for every }x\in (0,2\kappa_{\min}\delta_1) \text{ we have }\sqrt{1-x}\ge 1-\alpha x.
\end{equation}

By Lemma \ref{small disc} there exists a non constant $u_0\in \OBZ$ such that $u_0(0)=p$ and $u_0(\overline{\D})\subset \Omega$. Moreover, without loss of generality we may assume that $p_0\notin u_0(b\D)$. Choose
$\tau\in(\rho(p_0),\min\{\rho(u_0(\zeta)),\zeta \in b\D\})$, $\tau<-\delta_1$.
By strict convexity of $\rho$ and compactness of $\{z\in\Omega\colon \rho(z)\in [\tau,-\delta_1]\}$, there is $\lambda>0$ such that for each $q\in\Omega$
with $ \rho(q)\in [\tau,-\delta_1]$ and for each $V\in T_q \Omega_{\rho(q)}$, $|V|=\frac {\delta_1} 2$,
we have $q+V\in \Omega$, and $\rho(q+V)> \rho(q)+\lambda$. That is, we have fixed a compact subset of $\Omega$ containing $u_0(\partial\D)$ such that any vector $V$ of length $\frac {\delta_1} 2$ that is tangent to any level set of $\rho$  within this set does not leave $\Omega$, and moreover, its endpoint is closer  to the
boundary  $b\Omega$ for at least $\lambda$.

We now push the boundary of $u_0$ closer to $b\Omega$. In particular, we construct a solution of the Pascali system $u_1\in \OBZ$ such that
\begin{equation}\label{equ1}
u_1(0)=p,\ u_1(\overline\D)\subset \Omega, \ \dist(u_1(\zeta),\partial \Omega)<\delta_1\text{ for all }\zeta \in b\D.
\end{equation}
This can be achieved in finitely many steps of the same kind, we explain the first one:
we choose a continuous non vanishing section $V$ of the complex vector bundle $T_{u_0(\zeta)}\Omega_{\rho(u_0(\zeta))}$ over $b\D$. By multiplying 
$V$ with a continuous non vanishing function we may assume that $|V|\le \frac{\delta_1}2 $, and that we have:
\begin{eqnarray}
|V(\zeta)|=\frac{\delta_1}2  \text{ for all }\zeta  \in b\D \text{ such that } \dist(u_0(\zeta),b\Omega)\ge \delta_1, \\ 
u_0(\zeta)+\xi V(\zeta)\in\Omega \text{ for all }\xi \in \overline\D \text{ and } \zeta \in b\D. \label{fm1}
\end{eqnarray}
We can extend $V$ continuously to $\overline\D$ so that (\ref{fm1}) holds for all $\zeta \in \overline\D$.
We apply Lemma \ref{RHlemma} for $\epsilon>0$ chosen small enough 
to obtain a map $w\in \OBZ$ with the following properties:
\begin{enumerate}
\item[(a1)]  $w(\overline\D)\subset \Omega$,
\item[(a2)]  $\rho(w(\zeta))>\rho(u_0(\zeta))+\lambda $ for all $\zeta\in b \D$ with $\dist(u_0(\zeta),b\Omega)\ge \delta_1$,
\item[(a3)] $\rho(w(\zeta))>\rho(u_0(\zeta))$  for all $\zeta\in b \D$,
\item[(a4)]  $w(0)=p$.
\end{enumerate}
In finitely many steps of this kind we obtain a solution to the Pascali system $u_1$
satisfying \eqref{equ1}.

In the second part, we construct inductively a sequence of solutions to the Pascali system $u_j\in \OBZ$,
 sequences  $\epsilon_j, \delta_j, \lambda_j, r_j$ such that:
\begin{enumerate}
\item[(b1)] $u_j(\overline\D)\subset \Omega$,
\item[(b2)] $\dist(u_j(\zeta),b\Omega)<\delta_{j}$  for $\zeta \in b\D$, 
\item[(b3)] $\lambda_{j-1}<\dist(u_j(\zeta),b\Omega)<\delta_{j-1}$ for  $\zeta$, $|\zeta|\in [r_{j-1},1]$,
\item[(b4)] for $z\in\Omega$ such that $\lambda_{j-1}<\dist(z,b\Omega)<\delta_{j-1}$ , and  $w\in\C^n$ such that  $|z-w|< {\epsilon_j}$ 
we have  $w\in\Omega$.
\item[(b5)] $r_j\in(\max\{1-\frac 1{2^{j}},r_{j-1}\},1)$ and $|u_{j+1}(\zeta)-u_j(\zeta)|<\frac{\epsilon_j}{2^j}$ for $\zeta$, $|\zeta|\le r_j$,
\item[(b6)] $|u_{j+1}(\zeta)-u_j(\zeta)|<\sqrt{2c\delta_{j}}$ for $\zeta \in b\D$,
\item[(b7)] $\delta_j<d\cdot \delta_{j-1}$ and  $\epsilon_{j-1}<\frac  1{2^{j-1}}$, 
\item[(b8)] $u_j(0)=p$.
\end{enumerate}

Note that $u_1$  and $\delta_1$ satisfy the properties (b1), (b2) and (b8) for $j=1$ by \eqref{equ1}.
For any $\delta_{-1}>\delta_0>\delta_1$, and $\epsilon_0>0$ satisfying (b7) for $j=1$ we may choose $0<r_0<1$ and $\lambda_0>0$ such that (b3) holds for $j=1$, 

Assume that for some $k\in\N$ we have already constructed $u_j$ and $\delta_j$ for $j=1,\ldots, k$, and $r_j$, 
$\lambda_j$, and $\epsilon_j$ for  $j=1,\ldots, k-1$ such that (b1), (b2), (b3), (b7) and (b8) hold for $j=1,\ldots, k$, and (b4), (b5), (b6) hold for $j=1,\ldots, k-1$.
For each $\zeta\in b\D$  let $q(\zeta)$ be the unique nearest point to $u_k(\zeta)$ in the boundary of $\Omega$.
Similarly to the first part, we choose a continuous non vanishing section $V_k$ of the complex vector bundle $T_{u_k(\zeta)}\Omega_{\rho(u_k(\zeta))}$ over $b\D$.
Multiplying  $V_k$ with a continuous non vanishing function we may assume that 
$|q(\zeta)+c\nu_{q(\zeta)} -(u_k(\zeta)+V_k(\zeta))|=c$. Recall that $c<\frac{1}{\kappa_{\max}}$, hence
$$\B\left(q(\zeta)+c\nu_{q(\zeta)},c\right)\subset \B\left(q(\zeta)+\frac{1}{\kappa_{\max}}\nu_{q(\zeta)},\frac{1}{\kappa_{\max}}\right)\subset\Omega.$$
This implies that $u_k(\zeta)+\xi\cdot V_k(\zeta)\in \Omega$ for all $\xi \in \overline\D$, $\zeta \in b\D$.
Moreover, by Pitagora's theorem for the triangle with vertices $u_k(\zeta)$, $u_k(\zeta)+\xi \cdot V_k(\zeta)$, $q(\zeta)+c\nu_{q(\zeta)}$, 
and by (b2) for $j=k$ we get 
\begin{equation}\label{equV_j}
|V_k(\zeta)|=\sqrt{c^2-(c-\dist(u_k(\zeta),b\Omega))^2}<\sqrt{2c\delta_{k}} \text{ for }\zeta \in b\D.
\end{equation}
Since we move in the direction tangent to $\Omega_{\rho(u_k(\zeta))}$, by convexity we also have
\begin{equation}\label{note}
\dist(u_k(\zeta)+\xi\cdot V_k(\zeta),b\Omega)\le\dist(u_k(\zeta),b\Omega) \text{ for }\xi\in \overline\D, \zeta \in b\D.
\end{equation}
Moreover, we can extend $V_k$ continuously to $\overline\D$ so that 
\begin{equation}\label{note2}
u_k(\zeta)+\xi \cdot V_k(\zeta)\in\Omega  \text{ for } \xi \in \overline\D, \zeta \in \overline\D.
\end{equation}

Let us estimate the distance between the point $u_k(\zeta)+\xi \cdot V_k(\zeta)$ $\xi \in b\D$, $\zeta \in b\D$, and $b\Omega$. By \eqref{ukriv} this distance is smaller or equal to the distance between the same point and the set $b\B\left(q(\zeta)+\textstyle{\frac{1}{\kappa_{\min}}}\nu_{q(\zeta)}\right)$. By Pitagora's theorem for the triangle with vertices $u_k(\zeta)$, $u_k(\zeta)+\xi \cdot V_k(\zeta)$, $q(\zeta)+1/\kappa_{\min}\nu_{q(\zeta)}$ this implies the following estimate
$$
\dist(u_k(\zeta)+\xi V_k(\zeta), b\Omega)
\leq \textstyle{\frac{1}{\kappa_{\min}}} -\sqrt{( \textstyle{\frac{1}{\kappa_{\min}}} -\dist(u_k(\zeta), b\Omega)
)^2+|V_k(\zeta)|^2}. $$
$$\leq\textstyle{\frac{1}{\kappa_{\min}}} \left(1-\sqrt{1-2\kappa_{\min}\dist(u_k(\zeta), b\Omega)+2\kappa_{\min}^2\dist(u_k(\zeta), b\Omega)|V_k(\zeta)|}\right),$$
 where in the last inequality we apply $a^2+b^2\geq 2ab$. By \eqref{defin_delta0} it follows that
$$\dist(u_k(\zeta)+\xi V_k(\zeta), b\Omega)\leq 2\alpha\cdot\dist(u_k(\zeta), b\Omega)(1-\kappa_{\min}|V_k(\zeta)|),$$
while from (b2) and \eqref{defin_d} we get 
$$\dist(u_k(\zeta)+\xi V_k(\zeta), b\Omega)\leq 2\alpha\cdot \delta_k (1-\kappa_{\min} c)=d\delta_k.$$
We can choose $0<\delta_{k+1}<d \delta_k$ such that 
\begin{equation}\label{note2.5}
\dist(u_k(\zeta)+\xi V_k(\zeta), b\Omega)<\delta_{k+1}  \text{ for } \xi \in b\D, \zeta \in b\D.
\end{equation}

For sufficiently small $\epsilon_k>0$ to be chosen later, we apply Lemma \ref{RHlemma} for $V=V_k$, $u=u_k$, $r_0=\max\{1-\frac 1{2^{k-1}},r_{k-1}\}$, $\epsilon=\frac {\epsilon_k}{2^k}$
and obtain the map $u_{k+1}=w$ and $r_k=r'\in (r_0,1)$. 
The property (b8) for $j=k+1$ follows from (iv) and
the property (b5)  for $j=k$ follows from  (iii).
For all sufficiently small $\epsilon_k>0$ 
the property (b1) for $j=k+1$ follows from \eqref{note2}, (ii) and (iii);
the property (b2)  for $j=k+1$ follows from \eqref{note2.5} and (i);
for any $\lambda_k>0$ small enough the property (b3)  for $j=k+1$ follows from \eqref{note} and (ii); 
the property (b6)  for $j=k$ follows from \eqref{equV_j} and (i).
We choose $\epsilon_k\in (0,\frac 1{2^{k}})$ so small that all of the above holds, and furthermore, that 
the property (b4)  is satisfied for $j=k$.
Then the property (b7) holds for $j=k+1$ and this completes the inductive step. 

Fix an integer $j>0$ and take $|\zeta|\le r_j$. Then for any integer $k>0$ we have by (b5): 
\begin{eqnarray*}
|u_{j}(\zeta)-u_{j+k}(\zeta)|&\le& |u_{j}(\zeta)-u_{j+1}(\zeta)|+\cdots +|u_{j+k-1}(\zeta)-u_{j+k}(\zeta)|
 \\
&\le& \frac{\epsilon_j}{2^j}+\cdots + \frac{\epsilon_{j+k-1}}{2^{j+k-1}}\le \epsilon_j <\frac 1{2^{j}}
\end{eqnarray*}
Therefore, the sequence $u_j$ converges uniformly on compacta on $\D$, and for the limit map  $u$
it holds
\begin{equation}\label{epsilonj}
|u_{j}(\zeta)-u(\zeta)|\le \epsilon_j \text { for all }|\zeta|\le r_j.
\end{equation}
By (b3), (b4), and \eqref{epsilonj}  we get that $u(\zeta)\in\Omega$. 
This implies that $u(\D)\subset \Omega$. 
By (b8) we have $u(0)=p$.
By (b6) and (b7) we get 
$$|u_{j+1}(\zeta)-u_j(\zeta)|\le \sqrt{2c\delta_{1}}(\sqrt{d})^{j-1}  \text{ for all } \zeta \in b\D.$$
Since $d<1$ (see \eqref{defin_d})
this further implies that the sequence $u_j$ is uniformly convergent on $b\D$, and by (b2) we have $u(b\D)\subset b\Omega$. We have proved that the sequence $u_n\in\OBZ$ converges uniformly on compact set in $\D$ and uniformly on $b\D$, thus by Lemma \ref{lemmakonv} the limit map $u$ lies in $\OBZ$.
Therefore, the map $u$ satisfies all properties in the theorem.
\qed

\end{document}